\documentclass[12pt]{article}
\usepackage[utf8]{inputenc}\usepackage[english]{babel}
\usepackage[T1]{fontenc}\usepackage{amsmath}
\usepackage{amssymb}\usepackage{amsthm}
\usepackage{amsfonts}\usepackage{mathrsfs}
\usepackage{geometry}\usepackage{enumerate}

\DeclareMathOperator{\convw}{\xrightarrow[]{w}}
\DeclareMathOperator{\convws}{\xrightarrow[]{w^{\ast}}}
\DeclareMathOperator{\convn}{\xrightarrow[]{\|\cdot\|}}

\newtheorem{theorem}{Theorem}
\newtheorem{definition}{Definition}
\newtheorem{example}{Example}

\newtheorem{lemma}{Lemma}
\newtheorem{corollary}{Corollary}
\makeatletter\renewcommand{\subsection}{\@startsection{subsection}{1}
{0pt}{3.25ex plus 1ex minus.2ex}{-1em}{\normalfont\normalsize\bf}}\makeatother

\begin{document}

\title{On compact (limited) operators between Hilbert and Banach spaces}
\author{Svetlana Gorokhova\\ 
\small Uznyj Matematiceskij Institut VNC RAN, Vladikavkaz, Russia}
\maketitle

\begin{abstract}
{We give orthonormal characterizations of 
collectively compact (limited) sets of linear operators 
from a Hilbert space to a Banach space.}
\end{abstract}



\smallskip
\smallskip

It is well known that a bounded linear operator between 
Hilbert spaces is compact if and only if it maps every orthonormal 
sequence onto a norm-null sequence (cf. \cite[p.\,95 and p.\,293]{Halmos1982}).
In Theorem \ref{main theorem1}, we extend this fact to linear operators from a
Hilbert space to a Banach space. More generally, in Theorem \ref{main theorem3}
we prove that a sets of linear operators from a
Hilbert space to a Banach space is collectively compact (resp., limited) if and only if 
the union of images of each orthonormal sequence is relatively compact (resp., limited).
In the sequel, $\mathcal{H}$ denotes a real or complex Hilbert space. 
For unexplained terminology and notation, we refer to \cite{A1971,AP1968,AB2006,BD1984,Halmos1982}.

\medskip
We need the following two elementary and certainly well known lemmas. 
For convenience, we include the proof of the second one, whereas
the proof of the first lemma is left as an exercise.

\begin{lemma}\label{main lemma}
{\em
Every weakly null sequence $(x_n)$ in $\mathcal{H}$ contains 
a subsequence  $(x_{n_k})$ such that 
$|(x_{n_{k_1}},x_{n_{k_2}})|\le 2^{-2(n_{k_1}+n_{k_2})}$ 
whenever $k_1\ne k_2$.
}
\end{lemma}

\begin{lemma}\label{bounded lemma}
{\em
A linear operator $T:\mathcal{H}\to Y$ from a 
Hilbert space $\mathcal{H}$ to a Banach space $Y$ is bounded if and only if $T$ 
maps orthonormal sequences onto bounded sequences.
}
\end{lemma}

\begin{proof}
Only the sufficiency needs a proof. Let a linear operator $T:\mathcal{H}\to Y$ be 
bounded on each orthonormal sequence of $\mathcal{H}$.
Suppose $T$ is not bounded. Then, there exists a linearly independent normalized 
sequence $(x_n)$ in $\mathcal{H}$ such that $\|Tx_n\|\ge 2^n$ for all $n$.
Pick (by the Gram--Schmidt orthonormalization) an orthonormal sequence
$(u_n)$ in $\mathcal{H}$ such that $x_n\in\text{\rm span}\{u_i\}_{i=1}^n$ for all $n$.
By the assumption, there exists an $M\in\mathbb{R}$ satisfying $\|Tu_n\|\le M$ 
for all $n\in\mathbb{N}$. Then 
$$
   2^n\le\|Tx_n\|=\Bigl\|\sum\limits_{k=1}^{n}(x_n,u_k)T(u_k)\Bigl\|\le nM 
   \ \ \ \ (\forall n \in\mathbb{N}),
$$
which is absurd. The obtained contradiction completes the proof.
\end{proof}

\bigskip
Recall that a subset $A$ of a normed space $X$ is limited whenever
$f_n\rightrightarrows 0(A)$ for every w$^\ast$-null sequence $(f_n)$ of $X'$.
Each limited set is necessarily bounded. Indeed, otherwise let $A$ be a limited subset of 
a normed space $X$ and $\|a_n\|\ge n$ for a sequence $(a_n)$ in $A$.
Take a sequence $(f_n)$ in $X'$ such that $f_n(a_n)=1$ and $\|f_n\|=\|a_n\|^{-1}$
for all $n$. Then $\|f_n\|\to 0$, and hence $f_n\convws 0$. As $A$ is limited,
$f_n\rightrightarrows 0(A)$ in contrary with $f_n(a_n)=1$ for all $n$.
Therefore, the next corollary follows from Lemma \ref{bounded lemma}.

\begin{corollary}\label{bounded cor}
{\em
A linear operator $T:\mathcal{H}\to Y$ from a Hilbert space $\mathcal{H}$ 
to a Banach space $Y$ is bounded 
if it maps orthonormal sequences onto limited sets.
}
\end{corollary}

\bigskip
The following definition goes back to \cite{A1971,AP1968}. 

\begin{definition}\label{col comp and col lim}
{\em
A subset ${\cal T}$ of the space $\text{\rm L}(X,Y)$ of linear operators between
normed spaces $X$ and $Y$ is 
\begin{enumerate}[$a)$]
\item
collectively compact if the set ${\cal T}B_X=\bigcup\limits_{T\in{\cal T}}T(B_X)$ 
is relatively compact.
\item
collectively limited if ${\cal T}B_X$ is limited.
\end{enumerate}
}
\end{definition}

\noindent
It follows immediately that an operator $T$ is compact (resp., limited) if and only if the set $\{T\}$ is
collectively compact (resp., collectively limited). The following theorem is the main result of the paper.

\begin{theorem}\label{main theorem3}
{\em
Let ${\cal T}\subseteq\text{\rm L}(\mathcal{H},Y)$,
where $\mathcal{H}$ is a Hilbert space and $Y$ is a Banach space.
Then the following holds.
\begin{enumerate}[$i)$]
\item
${\cal T}$ is collectively compact if and only if the set $\{Tx_n: T\in{\cal T}; n\in\mathbb{N}\}$
is relatively compact for every orthonormal sequence $(x_n)$ in $\mathcal{H}$.
\item
${\cal T}$ is collectively limited if and only if the set $\{Tx_n: T\in{\cal T}; n\in\mathbb{N}\}$
is limited for every orthonormal sequence $(x_n)$ in $\mathcal{H}$.
\end{enumerate}
}
\end{theorem}

\begin{proof}
We may assume $\text{\rm dim}(\mathcal{H})=\infty$.
Only the sufficiency requires proof.
First, observe that $\sup\limits_{T\in{\cal T}}\|T\|\le M<\infty$
in both of cases. Indeed, otherwise $\sup\limits_{T\in{\cal T}}\|Tx\|=\infty$ for some
$x\in\mathcal{H}$, $\|x\|=1$ by the Uniform Boundedness Principle,
which leads to contradiction via taking an orthonormal sequence $(x_n)$ 
in $\mathcal{H}$ with $x_1=x$.

\medskip
$i)$\
Suppose, in contrary, $\bigcup\limits_{T\in{\cal T}}T(B_\mathcal{H})$ 
is not relatively compact. Then, for some $\alpha>0$, there exist a normalized 
sequence $(y_n)$ in $\mathcal{H}$ and a sequence $(T_n)$ in ${\cal T}$ satisfying 
$\|T_ny_n-T_my_m\|\ge 3\alpha$ for all $m\ne n$. 

There are two mutually exclusive cases. 
Case~{\bf (A)}: $(y_n)$ has no  norm-convergent subsequence.
Case~{\bf (B)}: $(y_n)$ has a norm-convergent subsequence.
Consider them separately.

\medskip

{\bf (A)}.
Since $B_\mathcal{H}$ is relatively weakly compact, 
by passing to subsequence, we may assume $y_n\convw y$
and $\|y_n-y_m\|\ge C>0$ for all $m\ne n$ and some $C$. 
By the weak lower semicontinuity of the norm (cf. \cite[Lemma\,6.22]{ABo2006}), 
$\|y_n-y\|\ge  C$ for all $n$. Let $z_n=\frac{y_n-y}{\|y_n-y\|}$. Then $z_n\convw 0$.
By the assumption, $(T_{n_j}y)$ is norm-Cauchy for some subsequence $(T_{n_j})$.
We may assume $\|T_{n_j}y-T_{n_i}y\|\le\min(j,i)^{-1}$ for all $j,i\in\mathbb{N}$. 
Then, for all $j\ne i$,
$$
   2\|T_{n_j}z_{n_j}-T_{n_i}z_{n_i}\|=
   2\left\|\frac{T_{n_j}(y_{n_j}-y)}{\|y_{n_j}-y\|}-\frac{T_{n_i}(y_{n_i}-y)}{\|y_{n_i}-y\|}\right\|\ge
   \|T_{n_j}(y_{n_j}-y)-T_{n_i}(y_{n_i}-y)\|\ge
$$
$$
   \|T_{n_j}y_{n_j}-T_{n_i}y_{n_i}\|-\|T_{n_j}y-T_{n_i}y\|\ge
   3\alpha-\|T_{n_j}y-T_{n_i}y\|\ge 3\alpha-\min(j,i)^{-1}.
$$
By applying Lemma \ref{main lemma} and passing to a further subsequence, we may assume 
that $|(z_{n_j},z_{n_i})|\le 2^{-2(n_j+n_i)}$ and 
$\|T_{n_j}z_{n_j}-T_{n_i}z_{n_i}\|\ge\alpha$ for all $j\ne i$.
By the Gram--Schmidt orthonormalization, there exists an orthonormal $(x_{n_j})$ 
in $\mathcal{H}$ with $\|x_{n_j}-z_{n_j}\|\le 2^{-n_j}$ for all $j$. 
By the compactness assumption, there exists a subsequence $(x_{n_{j_l}})$ such that
$T_{n_{j_l}}x_{n_{j_l}}\convn w\in Y$, and hence
$$
   \alpha\le\|T_{n_{j_l}}z_{n_{j_l}}-T_{n_{j_p}}z_{n_{j_p}}\|\le
   \|T_{n_{j_l}}x_{n_{j_l}}-T_{n_{j_p}}x_{n_{j_p}}\|+
   \|T_{n_{j_l}}(x_{n_{j_l}}-z_{n_{j_l}})\|+\|T_{n_{j_l}}(x_{n_{j_p}}-z_{n_{j_p}})\|\le
$$
$$
   \|T_{n_{j_l}}x_{n_{j_l}}-w\|+\|T_{n_{j_p}}x_{n_{j_p}}-w\|+
   (2^{-{n_{j_l}}}+2^{-{n_{j_p}}})M\to 0 \ \ \ \ \ \ (l,p\to\infty).
$$
A contradiction.

\medskip

{\bf (B)}. Assume $(y_n)$ has a norm-convergent subsequence,
say $y_{n_j}\convn y$. Let $\|y_{n_j}-y\|\le M^{-1}\alpha$ for $j\ge j_0$.
Then, for all $j,i\ge j_0$,
$$
   \|T_{n_j}y-T_{n_i}y\|=\|T_{n_j}y-T_{n_j}y_{n_j}+T_{n_j}y_{n_j}-
   T_{n_i}y_{n_i}+T_{n_i}y_{n_i}-T_{n_i}y\|\ge
$$
$$
   -\|T_{n_j}y-T_{n_j}y_{n_j}\|+\|T_{n_j}y_{n_j}-T_{n_i}y_{n_i}\|-\|T_{n_i}y_{n_i}-T_{n_i}y\|\ge
$$
$$
   3\alpha-M(\|y_{n_j}-y\|+\|y_{n_i}-y\|)\ge 3\alpha-M(M^{-1}\alpha+M^{-1}\alpha)=\alpha,
$$
which leads to a contradiction by taking an orthonormal sequence $(z_n)$ with $z_1=y$.
It follows that $\bigcup\limits_{T\in{\cal T}}T(B_\mathcal{H})$ is relatively compact.

\medskip
$ii)$\
Suppose in contrary $\bigcup\limits_{T\in{\cal T}}T(B_\mathcal{H})$ 
is not limited. Then, for some $\ast$-weakly null sequence
$(f_n)$ in $Y'$ there exist an $\alpha>0$, a sequence $(T_n)$ in ${\cal T}$, and 
a normalized sequence $(x_n)$ in $\mathcal{H}$ satisfying $|f_n(T_nx_n)|\ge\alpha$ for all $n$.

Let $0\ne z\in\mathcal{H}$. Take any orthonormal sequence $(z_n)$ with $z_1=\frac{z}{\|z\|}$.
Since $|f_n(T_nz)|\le\sup\limits_{k\in\mathbb{N}; T\in{\cal T}}|f_n(Tz_k)|$ for all $n$,
and $\lim\limits_{n\to\infty}\sup\limits_{k\in\mathbb{N}; T\in{\cal T}}|f_n(Tz_k)|=0$
by the assumption, then $|(T'_nf_n)z|=|f_n(T_nz)|\to 0$.
So, the sequence $(T'_nf_n)$ is $\ast$-weakly null in $\mathcal{H}'$,
and hence (by identifying $\mathcal{H}$ with $\mathcal{H}''$) the sequence 
$(T'_nf_n)$ is weakly null in the Hilbert space $\mathcal{H}'$.
Since $|(T'_nf_n)x_n|=|f_n(T_nx_n)|\ge\alpha$ then $\|T'_nf_n\|\ge\alpha$ for all $n$. 
Thus, by scaling $(f_n)$, we may assume $\|T'_nf_n\|=1$ for all $n$.

Lemma \ref{main lemma} implies the existence of a subsequence
$(T'_{n_k}f_{n_k})$ satisfying  
$$
   |(T'_{n_{k_1}}f_{n_{k_1}},T'_{n_{k_2}}f_{n_{k_2}})|\le 2^{-2(n_{k_1}+n_{k_2})}
   \ \ \ \ \ \ \ (k_1\ne k_2).
$$ 
Applying the Gram--Schmidt orthonormalization to $(T'_{n_k}f_{n_k})$, 
we obtain an orthonormal sequence $(z_{n_k})$ in $\mathcal{H}'$
such that $\|z_{n_k}-T'_{n_k}f_{n_k}\|\le 2^{-k}$ for all $k$. 
Pick a bi-orthogonal for $(z_{n_k})$ sequence $(y_{n_{k}})$ in $\mathcal{H}$,
i.e, $(y_{n_{k}})$ is orthonormal and $z_{n_{k}}(y_{n_{k}})=1$ for every $k$.
By the assumption, $\{Ty_{n_{l}}\}_{l\in\mathbb{N}; T\in{\cal T}}$ is a limited subset of $Y$.
Thus, $f_n\rightrightarrows 0\left(\{Ty_{n_{l}}\}_{l\in\mathbb{N}; T\in{\cal T}}\right)$, and hence
$T'_nf_n\rightrightarrows 0\big(\{y_{n_{l}}\}_{l=1}^\infty\big)$ violating
$$
   \limsup_{n\to\infty}\left(\sup\limits_{l\in\mathbb{N}}\big|(T'_nf_n)y_{n_{l}}\big|\right)\ge
   \limsup_{k\to\infty}\left(\sup\limits_{l\in\mathbb{N}}
   \big|(T'_{n_{k}}f_{n_{k}})y_{n_{l}}\big|\right)\ge
$$
$$
   \limsup_{k\to\infty}\big|(T'_{n_{k}}f_{n_{k}})y_{n_{k}}\big|=
   \limsup_{k\to\infty}\big|z_{n_{k}}(y_{n_{k}})\big|=1.
$$
The obtained contradiction completes the proof.
\end{proof}

\bigskip
Next, we give an orthonormal characterization of compact operators. 

\begin{theorem}\label{main theorem1}
{\em
Let $T:\mathcal{H}\to Y$ be a linear operator 
from a Hilbert space $\mathcal{H}$ to a Banach space $Y$. 
The following conditions are equivalent.
\begin{enumerate}[$i)$]
\item
$T$ is compact.
\item
$T$ maps every orthonormal subset of $\mathcal{H}$ into a relatively compact set.
\item
$T$ maps every orthonormal sequence of $\mathcal{H}$ to a norm-null sequence.
\item
A set $A_\varepsilon=\{a\in A:\|Ta\|\ge\varepsilon\}$ is finite for every orthonormal 
basis $A$ of $\mathcal{H}$ and every $\varepsilon>0$.
\item
$T$ maps every orthonormal basis of $\mathcal{H}$ into a relatively compact set.
\end{enumerate}
}
\end{theorem}

\begin{proof}
We may assume $\text{\rm dim}(\mathcal{H})=\infty$.

\medskip
$i)\Longrightarrow iii)$:\
Let $(x_n)$ be an orthonormal sequence in $\mathcal{H}$. 
Then $x_n\stackrel{\rm w}{\to}0$, and hence $\|Tx_n\|\to 0$ because $T$ is compact.

\medskip
$iii)\Longrightarrow i)$:\
It follows from Theorem \ref{main theorem3}~$i)$.

\medskip
$ii)\Longrightarrow iii)$:\
The operator $T$ is bounded by Lemma \ref{bounded lemma}.
Let $(u_n)$ be an orthonormal sequence of $\mathcal{H}$.
On the way to a contradiction, suppose $\|Tu_n\|\not\to 0$.
By passing to a subsequence, we may suppose $\|Tu_n\|\ge M$ for all $n$ and some $M>0$.
By $i)$, the set $\{Tu_n:n\in\mathbb{N}\}$ is relatively compact,
and hence $(Tu_n)$ has a norm-convergent subsequence,
say $\lim\limits_{j\to\infty}\|Tu_{n_j}-y\|=0$ for some $y\in Y$. 
As $(u_{n_j})$ is orthonormal, $u_{n_j}\convw 0$. Since $T$ is bounded,
$Tu_{n_j}\convw 0$. As $Tu_{n_j}\convn y$, we obtain $y=0$, and hence
$(Tu_{n_j})$ is norm-null, which is absurd since $\|Tu_n\|\ge M>0$ for all $n$.

\medskip
$iii)\Longrightarrow ii)$:\
Let $A$ be an orthonormal subset of $\mathcal{H}$. 
Pick a sequence $(y_n)$ in $T(A)$. We need to show that
$y_{n_j}\convn y$ for some subsequence $(y_{n_j})$ and
some $y\in Y$. Clearly, $(y_n)$ has a norm-convergent subsequence
when some term $y_n$ of $(y_n)$ occurs infinitely many times. 
So, suppose each term of the sequence $(y_n)$ occurs at most
finitely many times. Keeping the first of occurrences of each $y_n$ in $(y_n)$
and removing others, we obtain a subsequence $(y_{n_j})$ whose terms are distinct.
Pick an $x_n\in A$ with $y_n=Tx_n$ for each $n\in\mathbb{N}$.
Then $(x_{n_j})$ is an orthonormal sequence of $\mathcal{H}$,
and hence $y_{n_j}=Tx_{n_j}\convn 0$ by $iii)$.

\medskip
Implications $iii)\Longrightarrow iv)\Longrightarrow v)$ are trivial, while
$v)\Longrightarrow ii)$ is a consequence of the fact that each 
orthonormal $S\subseteq\mathcal{H}$ has an extention to an orthonormal basis.
\end{proof}

\noindent
The condition $v)$ 
of Theorem \ref{main theorem1} cannot be replaced by the condition that
$T$ maps some orthonormal basis of $\mathcal{H}$ into a relatively compact set 
(see, for example \cite[p.292]{Halmos1982}).

\bigskip
Recall that a linear operator $T:X\to Y$ between normed spaces is limited 
if $T(B_X)$ is limited in $Y$, where $B_X$ is the closed unit ball of $X$. 
It follows immediately that a bounded linear operator $T$ is limited if
$T'$ carries w$^\ast$-null sequences to norm-null ones.
It is well known that compact operators from $\mathcal{H}$ to $Y$ agree with limited operators
whenever $Y$ is reflexive or separable~\cite{BD1984}. In general, they are different.

\begin{example}{\em
Let $\mathcal{B}$ be an ortonormal basis in an infinite dimensional Hilbert space $\mathcal{H}$, 
and let $Y=\ell^\infty(\mathcal{B})$ be the space of bounded scalar functions on $\mathcal{B}$. 
The operator $T:\mathcal{H}\to Y$ defined by
$(Tx)(u)=(x,u)$ for $x\in\mathcal{H}$ and $u\in\mathcal{B}$ is bounded yet not compact
(the $u$-th coordinate of $Tu$ is one whereas all others are zeros).
Since $\mathcal{B}$ is an orthonormal basis, 
$T(B_{\mathcal{H}})\subseteq B_{c_0(\mathcal{B})}$.
It follows from Phillip’s lemma (cf. \cite[Theorem~4.67]{AB2006}) that 
$B_{c_0(\mathcal{B})}$ is a limited subset of $Y$, and hence $T$ is limited.
}
\end{example}

\begin{theorem}\label{cor-lim}
{\em
Let $T:\mathcal{H}\to Y$ be a linear operator 
from an infinite dimensional Hilbert space $\mathcal{H}$ to a Banach space $Y$. 
The following conditions are equivalent.
\begin{enumerate}[$i)$]
\item
$T$ is limited.
\item
$T$ maps every orthonormal subset of $\mathcal{H}$ onto a limited set.
\item
$T$ maps every orthonormal basis of $\mathcal{H}$ onto a limited set.
\item
$T$ maps every orthonormal sequence of $\mathcal{H}$ onto a limited set.
\end{enumerate}
}
\end{theorem}

\begin{proof}
Implications $i)\Longrightarrow ii)\Longrightarrow iii)\Longrightarrow iv)$ are trivial,
and $iv)\Longrightarrow i)$ follows from Theorem \ref{main theorem3}~$ii)$.
\end{proof}

\noindent
The condition $ii)$ of Theorem \ref{cor-lim} cannot be replaced by the condition that
$T$ carries some orthonormal basis of $\mathcal{H}$ onto a limited set.
Indeed, limited sets in $\mathcal{H}$ agree with relatively compact sets \cite{BD1984}, and hence 
limited operators from $\mathcal{H}$ to $\mathcal{H}$
coincide with compact operators. So, apply again \cite[p.292]{Halmos1982}.
Clearly, the orthonormal sets, bases, and sequences 
in Theorems \ref{main theorem1} and \ref{cor-lim}
can be replaced by bounded orthogonal ones. The Banach lattice setting,
where the bounded disjointedness can be used instead of bounded orthogonality
is much more rigid, and it produces two new classes of disjointedly compact
and disjointedly limited operators \cite{EEG2024}.


\bibliographystyle{plain}
\end{document}